\documentclass{amsart}
\usepackage{amsrefs}
\usepackage[all]{xy}

\DeclareMathOperator{\val}{val}
\DeclareMathOperator{\ord}{ord}

\newcommand{\BP}{{\mathbf P}}
\newcommand{\BZ}{{\mathbf Z}}

\newcommand{\frakd}{{\mathfrak d}}
\newcommand{\frakD}{{\mathfrak D}}
\newcommand{\frakp}{{\mathfrak p}}
\newcommand{\frakq}{{\mathfrak q}}

\newcommand\lowtilde{\lower0.7ex\hbox{~}}

\newcommand{\mybar}[1]{#1\llap{$\overline{\phantom{\rm#1}}$}}
\renewcommand{\bar}{\mybar}

\newcommand{\nb}{\bar{n}}
\newcommand{\Cb}{\bar{C}}
\newcommand{\Db}{\bar{D}}
\newcommand{\Gb}{\bar{G}}

\newcommand{\nbar}{{\bar{\scriptstyle n}}}

\newcommand{\alphab}{\overline{\alpha}}
\newcommand{\betab}{\overline{\beta}}

\newcommand{\diffC}{{\frakd}_C}
\newcommand{\diffD}{{\frakd}_D}
\newcommand{\discC}{{\frakD}_C}
\newcommand{\discD}{{\frakD}_D}

\newtheorem{theorem}{Theorem}
\newtheorem{lemma}[theorem]{Lemma}

\theoremstyle{remark}
\newtheorem*{remark}{Remark}

\begin{document}

\title[Automorphisms of hyperelliptic curves]
      {Characteristic polynomials of\\ 
       automorphisms of hyperelliptic curves}

\author{Robert M. Guralnick}
\address{Department of Mathematics, 
         University of Southern California, 
         Los Angeles, CA 90089-2532, USA.}
\email{guralnic@usc.edu}
\urladdr{http://www-rcf.usc.edu/\lowtilde{}guralnic/}

\author{Everett W.~Howe}
\address{Center for Communications Research, 
         4320 Westerra Court, 
         San Diego, CA 92121-1967, USA.}
\email{however@alumni.caltech.edu}
\urladdr{http://alumni.caltech.edu/\lowtilde{}however/}

\date{2 April 2008}

\thanks{The first author was partially supported by
        NSF grant DMS 0653873.}



\begin{abstract}
Let $\alpha$ be an automorphism of a hyperelliptic curve $C$ of genus
$g$ and let $\alphab$ be the automorphism induced by $\alpha$ on the 
genus-$0$ quotient of $C$ by the hyperelliptic involution.
Let $n$ be the order of $\alpha$ and let $\nb$ be the order 
of~$\alphab$.  We show that the characteristic polynomial $f$ of the
automorphism $\alpha^*$ of the Jacobian of $C$ is determined by the 
values of $n$, $\nb$, and $g$, unless $n=\nb$, $\nb$ is even, and 
$(2g+2)/n$ is even, in which case there are two possibilities for~$f$. 
In every case we give explicit formulas for the possible 
characteristic polynomials.
\end{abstract}

\maketitle

\section{Introduction}
\label{S-intro}

Let $\alpha$ be an automorphism of a genus-$g$ curve $C$ over a
field~$k$ and let $\alpha^*$ be the corresponding automorphism of the
Jacobian of~$C$.  Let $n$ be the order of $\alpha$ and let $f$ be the 
characteristic polynomial of $\alpha^*$.  The values of $n$ and $g$ 
provide some restrictions on the possible values of $f$, but in general
they do not determine~$f$; for example, a nontrivial involution of a 
genus-$3$ curve can have characteristic polynomial equal to 
$(x-1)^i(x+1)^{6-i}$ for $i\in\{0,2,4\}$, and all three possibilities
occur.

If $C$ is hyperelliptic, with hyperelliptic involution $\iota$, 
then the automorphism $\alpha$ gives rise to an automorphism $\alphab$
of the genus-$0$ quotient $C/\langle\iota\rangle$.  Let $\nb$ be the
order of $\alphab$, so that either $n=\nb$ or $n=2\nb$.  The triple
$(g,n,\nb)$ still does not in general determine $f$: If $C$ has genus
$3$ and $\alpha$ and $\alphab$ each have order $2$, then $f$ can be
either $(x-1)^2(x+1)^4$ or $(x-1)^4(x+1)^2$, and both possibilities
occur.

The purpose of this note is to show that if $C$ is hyperelliptic, this
ambiguity between two possible characteristic polynomials is the worst
that can happen; furthermore, the triple $(g,n,\nb)$ determines $f$
completely unless $n=\nb$, $n$ is even, and $(2g+2)/n$ is an even
integer.

\begin{theorem}
\label{T:main}
Let $C$ be a hyperelliptic curve of genus $g$ over a field $k$ 
and let $\alpha$, $\alphab$, $n$, $\nb$, and $f$ be as above.  
\begin{enumerate}
\item\label{itemb} If $\nb$ is odd and $n=\nb$, then 
     $2g\equiv 0, -1,\text{\ or\ } -2\bmod \nb$, and
     \[ f = \begin{cases}
            \displaystyle \frac{(x^{\nbar} - 1)^{(2g+2)/\nbar} }{ (x - 1)^2} & \text{if $2g\equiv-2\bmod \nb$}\textup{;}\\[1em]
            \displaystyle \frac{(x^{\nbar} - 1)^{(2g+1)/\nbar} }{ (x - 1)  } & \text{if $2g\equiv-1\bmod \nb$}\textup{;}\\[1em]
                                (x^{\nbar} - 1)^{2g/\nbar}                   & \text{if $2g\equiv 0\bmod \nb$}.\\                  
            \end{cases}
     \]
\item\label{itema} If $\nb$ is odd and $n=2\nb$, then 
     $2g\equiv 0, -1,\text{\ or\ } -2\bmod \nb$, and
     \[ f = \begin{cases}
            \displaystyle \frac{(x^{\nbar} + 1)^{(2g+2)/\nbar} }{ (x + 1)^2} & \text{if $2g\equiv-2\bmod \nb$}\textup{;}\\[1em]
            \displaystyle \frac{(x^{\nbar} + 1)^{(2g+1)/\nbar} }{ (x + 1)  } & \text{if $2g\equiv-1\bmod \nb$}\textup{;}\\[1em]
                                (x^{\nbar} + 1)^{2g/\nbar}                   & \text{if $2g\equiv 0\bmod \nb$}.\\                  
            \end{cases}
     \]
\item\label{itemd} If $\nb$ is even and $n=\nb$, then 
     $2g\equiv-2\bmod \nb$.  Furthermore\textup{:}
     \begin{enumerate}
     \item\label{itemd1} if $(2g+2)/\nb$ is odd, then
          \[ f = \frac{(x^{\nbar} - 1)^{(2g+2)/\nbar} }{ (x^2 - 1)};
          \]
     \item\label{itemd2} if $(2g+2)/\nb$ is even, then
          \[ f = \frac{(x^{\nbar} - 1)^{(2g+2)/\nbar} }{ (x - 1)^2}
             \text{\quad or\quad }
             f = \frac{(x^{\nbar} - 1)^{(2g+2)/\nbar} }{ (x + 1)^2}.
          \]
     \end{enumerate}
\item\label{itemc} If $\nb$ is even and $n=2\nb$, then 
     $2g\equiv 0\bmod \nb$ and 
     \[ f = (x^\nbar+1)^{2g/\nbar}.\]
\end{enumerate}
\end{theorem}

\begin{remark}
Note that in Statements~\eqref{itemb} and~\eqref{itema} of the theorem,
if $\nb=1$ then the three expressions in the equality for $f$ are all
the same.
\end{remark}

\begin{remark}
The ambiguity in Statement~\eqref{itemd2} is unavoidable.  Suppose
$\alpha$ is an automorphism of $C$ for which $n=\nb$, $\nb$ is even,
and $(2g+2)/\nb$ is even.  Then $\alpha$ and $\iota\alpha$ give the
same values of $n$ and $\nb$, but they have different characteristic
polynomials.
\end{remark}

One of our motivations for the work in this paper was Proposition~13.1
of~\cite{HNR}, which is concerned with automorphisms $\alpha$ of 
supersingular genus-$2$ curves $C$ over finite fields of characteristic
at least~$5$.  The proposition says in part that if $\alpha$ is such an
automorphism, and if $n$ and $\nb$ are as defined above, then the pair
$(n,\nb)$ appears in the left-hand column of 
Table~\ref{Table:CharPolys}, and the characteristic polynomial of
$\alpha^*$ is as given in the right-hand column of the table.  
Here we note that Theorem~\ref{T:main} shows that the same conclusion
holds for automorphisms of arbitrary genus-$2$ curves over arbitrary
fields, with the restrictions on the values of $n$ and $\nb$ coming 
from the congruence conditions in the theorem.

\begin{table}
\begin{center}
\renewcommand\arraystretch{1.2}
\begin{tabular}{|c|c|}
\hline
$(n,\nb)$&                                \\ \hline\hline
$(1,1)$  & $(x - 1)^4$                    \\ \hline
$(2,1)$  & $(x + 1)^4$                    \\ \hline
$(2,2)$  & $(x - 1)^2(x + 1)^2$           \\ \hline
$(3,3)$  & $(x^2 + x + 1)^2$              \\ \hline
$(4,2)$  & $(x^2 + 1)^2$                  \\ \hline
$(5,5)$  & $x^4 + x^3 + x^2 + x + 1$      \\ \hline
$(6,3)$  & $(x^2 - x + 1)^2$              \\ \hline
$(6,6)$  & $(x^2 - x + 1)(x^2 + x + 1)$   \\ \hline
$(8,4)$  & $x^4 + 1$                      \\ \hline
$(10,5)$ & $x^4 - x^3 + x^2 - x + 1$      \\ \hline
\end{tabular}
\end{center}
\vspace{1ex}
\caption{Characteristic polynomials associated to possible values
of $n$ and $\overline{n}$ for genus-$2$ curves~\cite{HNR}*{Table 4}.}
\label{Table:CharPolys}
\end{table}

In Section~\ref{S:quotients} we prove two lemmas about quotients of
hyperelliptic curves by cyclic groups.  In Section~\ref{S:proof} we
use these lemmas to prove Theorem~\ref{T:main}.

\subsubsection*{Conventions}
In this paper, a \emph{curve} will always mean a 
geometrically-irreducible one-dimensional nonsingular scheme over a
field $k$; by the usual equivalence of categories, we could just as 
well phrase the entire paper in terms of one-dimensional function 
fields over~$k$.  When we speak of the projective line $\BP^1$ over a 
field~$k$, we will usually pick without comment a generator $x$ of its
function field, so that we can identify the function field with $k(x)$.

\section{Quotients of hyperelliptic curves}
\label{S:quotients}

Our proof of Theorem~\ref{T:main} will depend on two lemmas concerning
quotients of hyperelliptic curves, which we state and prove in this
section.  Throughout this section, $C$ will be a hyperelliptic curve
over an algebraically-closed field $k$, $\iota$ will be the 
hyperelliptic involution on $C$, and $\beta$ will be an automorphism
of $C$ of order $m$ such that $\iota\not\in\langle\beta\rangle$.

Let $D$ be the quotient of $C$ by the group $\langle\beta\rangle$.
Since $\iota$ is a central element of the automorphism group of $C$,
the automorphism $\beta$ induces an automorphism $\betab$ on the 
genus-$0$ curve $\Cb := C/\langle\iota\rangle$, and we get a diagram
\begin{equation}
\label{EQ:prop}
\xymatrix{
C \ar[d]^2 \ar[rr]^{\langle\beta\rangle} &&  D\ar[d]^2\\
\Cb\ar[rr]^{\langle\betab\rangle}    && \Db
}
\end{equation}
where the vertical arrows are quotients by $\langle\iota\rangle.$
Let $\varphi$ be the map from $C$ to $\Db$ and let $\psi$ be the map
from $\Cb$ to $\Db$.  We see that $\varphi$ and $\psi$ are both Galois
covers; the Galois group $G$ of $\varphi$ is generated by $\beta$ and 
$\iota$ and is isomorphic to $(\BZ/m\BZ)\times(\BZ/2\BZ)$, and the 
Galois group $\Gb$ of $\psi$ is generated by $\betab$ and is cyclic of
order $m$.  Note that $\Gb$ is the quotient of $G$ by 
$\langle\iota\rangle$.

\begin{lemma}
\label{L:ramgroups}
Let $Q$ be a point of $\Db$ and let $H$ be the 
inertia group of $Q$ in the cover~$\varphi$.
\begin{enumerate}
\item\label{modd}
      If $m$ is odd then $H$ is either the trivial group,
      the group $\langle\iota\rangle$,
      the group $\langle\beta\rangle$, or all of~$G$.
\item\label{meven}
      If $m$ is even and the characteristic of $k$ is not $2$, then
      $H$ is either the trivial group, the group $\langle\iota\rangle$,
      the group $\langle\beta\rangle$, or 
      the group~$\langle\iota\beta\rangle$.
\end{enumerate}
\end{lemma}

Before we begin the proof of the lemma we mention some facts about 
automorphisms of genus-$0$ curves that we will use repeatedly.

Suppose $\gamma$ is a finite-order automorphism of a genus-$0$ curve
$X$ over an algebraically-closed field~$k$.  By choosing an appropriate
isomorphism $X\cong\BP^1$, we may write the action of $\gamma$ on the
function field of $\BP^1$ in one of two forms: either $x\mapsto\xi x$
for a root of unity $\xi$, or $x\mapsto x+1$.  In the first case the
order of $\gamma$ is not divisible by the characteristic $p$ of $k$. 
Furthermore, the quotient map from $\BP^1$ to $\BP^1$ induced by 
$\gamma$ gives a Kummer extension of function fields $k(x)\to k(x)$ 
that can be written as $x\mapsto x^m$, where $m$ is the order
of~$\gamma$.  This map has two ramification points, and each point
ramifies totally.  When $\gamma$ can be written $x\mapsto x+1$ the 
order of $\gamma$ is equal to~$p$, and the associated quotient map 
$\BP^1\to\BP^1$ gives an Artin-Schreier extension of function fields 
$k(x)\to k(x)$ that can be written as $x\mapsto x^p-x$.  Only one point
of $\BP^1$ ramifies in this map, but again the ramification is total.

\begin{proof}[Proof of Lemma~\textup{\ref{L:ramgroups}}]
Suppose $m$ is odd, so that the Galois group $G$ is cyclic.  We know
that if a $Q$ ramifies in $\psi$, then it ramifies completely.  Thus,
the image of $H$ in $\Gb$ is either trivial or all of $\Gb$.  The only
subgroups of $G$ that have these images in $\Gb$ are the ones listed in
first statement of the lemma.

Suppose $m$ is even and the characteristic $p$ of $k$ is not $2$.
Since the automorphism $\betab$ of $\Cb$ has order $m$, and $m\neq p$,
the facts we mentioned before the start of the proof show that that $p$
does not divide~$m$.  Thus $p$ does not divide $\#G=2m$, so all 
ramification in $\varphi$ is tame.  In particular, the inertia group 
$H$ is cyclic.  The only cyclic subgroups of $G$ whose images in $\Gb$
are either trivial or all of $\Gb$ are the four groups listed in the
second statement.
\end{proof}

\begin{lemma}
\label{L:genera}
With notation and assumptions as above, let $g$ be the genus of $C$,
let $h$ be the genus of $D$, and let $e$ be the number of points of 
$\Db$ that ramify in both the right and the bottom maps of 
Diagram~\eqref{EQ:prop}.  Then $e\in\{0,1,2\}$.  If the characteristic
of $k$ is not $2$, then the relationship between $g$ and $h$ depends
on $e$ and on the parity of $m$ as follows\/\textup{:}
\smallskip
\begin{center}
\renewcommand{\arraystretch}{1.1}
\begin{tabular}{|c||l|l|}
\cline{2-3}
\multicolumn{1}{c||}{}
       & \quad $m$ odd                          & \quad $m$ even                         \\ \hline\hline
 $e=0$ & \quad $2h = (2g+2)/m - 2$ \hbox{\quad} & \quad $2h = (2g+2)/m - 2$ \hbox{\quad} \\ \hline
 $e=1$ & \quad $2h = (2g+1)/m - 1$              & \quad $2h = (2g+2)/m - 1$              \\ \hline
 $e=2$ & \quad $2h = 2g/m$                      & \quad $2h = (2g+2)/m$                  \\ \hline
\end{tabular}
\end{center}
\smallskip
If $k$ has characteristic~$2$, then $m$ is equal to $2$ if it is even,
and the relationship between $g$ and $h$ depends on $e$ and on the 
parity of $m$ as follows\/\textup{:}
\smallskip
\begin{center}
\renewcommand{\arraystretch}{1.1}
\begin{tabular}{|c||l|l|}
\cline{2-3}
\multicolumn{1}{c||}{}
       & \quad $m$ odd                          & \quad $m = 2$                                                        \\ \hline\hline
 $e=0$ & \quad $2h = (2g+2)/m - 2$ \hbox{\quad} & \quad $2h = g-1$                                                     \\ \hline
 $e=1$ & \quad $2h = (2g+1)/m - 1$              & \quad \rlap{$2h=g$}\phantom{$2h = g +1$} if $g$ is even,\hbox{\quad} \\ 
       &                                        & \quad $2h = g +1 $ if $g$ is odd                                     \\ \hline
 $e=2$ & \quad $2h = 2g/m$                      & \quad \textup{(}not possible\/\textup{)}                             \\ \hline
\end{tabular}
\end{center}
\end{lemma}

\begin{proof}
We know that at most two points ramify in the cover $\psi:\Cb\to \Db$,
so it follows immediately that $e$ is at most~$2$.

Let $\diffC$ and $\diffD$ denote the differents of the double covers
$C\to \Cb$ and $D\to \Db$, respectively, and let $\discC$ and $\discD$
be the discriminants of these covers.  Note that we have
$\deg\discC = \deg\diffC$ and $\deg\discD = \deg\diffD$.  More
specifically, if $P$ is a point of $\Cb$ at which $\discC$ has positive
order, then there is a unique point $\frakp$ of $C$ over~$P$, and 
$\ord_\frakp\diffC = \ord_P\discC$; the analogous statement holds for 
points of $\Db$.  The Riemann-Hurwitz 
formula~\cite{goldschmidt}*{Thm.~3.3.5}, applied to the double covers
$C\to \Cb$ and $D\to \Db$, shows that 
\begin{alignat*}{2}
g &= -1 + (1/2)\deg\diffC &&= -1 + (1/2)\deg\discC \\
\intertext{and}
h &= -1 + (1/2)\deg\diffD &&= -1 + (1/2)\deg\discD.
\end{alignat*}
Therefore, to find the relationship between $g$ and $h$ we need only
find the relationship between the degrees of $\discC$ and $\discD$.

Before we turn to the various cases summarized in the tables in the 
statement of the lemma, we will sketch out the general method we use to
compare the degrees of these two discriminants.  Throughout this 
introductory sketch, we will assume that we are not in the special case
where $m$ is even and $k$ has characteristic $2$.

Suppose $P$ is a point of $\Cb$ that ramifies in the double cover
$C\to \Cb$.  By looking at the lists in Lemma~\ref{L:ramgroups} of the
possible ramification groups for the point $\psi(P)$ in the extension 
$\varphi:C\to \Db$, we see that $\psi(P)$ must ramify in the double 
cover $D\to \Db$.  In other words, the support of $\discC$ is contained
in the inverse image under $\psi$ of the support of $\discD$.

We divide the support of $\discD$ into two sets: Let $E$ be the set of
points of $\Db$ that ramify both in $D\to \Db$ and in $\psi$, and let 
$E'$ be the set of points of $\Db$ that ramify in $D\to \Db$ but not 
in~$\psi$.  Then we have $e =\#E$, and we set $e' := \# E'$.  Let
$\discC'$ be the part of $\discC$ supported on $\psi^{-1}(E')$, and 
let $\discD'$ be the part $\discD$ supported on~$E'$.

Suppose $Q$ is a point of $E'$, and let $P$ be one of the $m$ points in
$\psi^{-1}(Q)$.  Then locally at $P$ and at $Q$ the extensions 
$C\to \Cb$ and $D\to \Db$ are isomorphic, so the order of $\discC$ at
$P$ is equal to the order of $\discD$ at $Q$.  This shows that
$\deg\discC' = m\deg\discD'$.

All that remains is to find the relationship between the portion of
$\discC$ supported on $\psi^{-1}(E)$ and the portion of $\discD$ 
supported on $E$.

Suppose $Q$ is a point of $E$.   For each $i$ let $H_i$ be the $i$-th
ramification group of $Q$ in the double cover $D\to \Db$.
By~\cite{goldschmidt}*{Thm.~3.5.9}, the order of $\discD$ at $Q$ is
equal to $\sum (\#H_i - 1)$, but since each $H_i$ has order $1$ or~$2$,
the value of this sum is simply the largest $i$ such that $H_i$ is
nontrivial.  Let $\frakq$ be the point if $D$ lying over $Q$, and let 
$v$ be a uniformizer at $\frakq$.  According to 
\cite{goldschmidt}*{Lem.~3.5.6}, the largest value of $i$ such that
$H_i$ is nontrivial is  the valuation of $v - \iota^*v$ at $\frakq$.
Thus, $\ord_Q \discD = \val_\frakq(v-\iota^*v)$.

Let $P$ be the unique point of $\Cb$ with $\psi(P) = Q$.  If $P$ is 
unramified in the double cover $C\to \Cb$ then $\discC$ has order $0$
at $P$.  If $P$ is ramified, let $\frakp$ be the point of $C$ lying 
over it, and let $u$ be a uniformizer at $\frakp$.  Arguing as above,
we find that $\ord_P \discC = \val_\frakp(u-\iota^*u)$.

With these formulas for $\ord_Q \discD$ and $\ord_P \discC$ in hand,
we turn to the various cases listed in the lemma.

First suppose that the characteristic of $k$ is not~$2$ and  that $m$
is odd.  If $Q$ is a point in~$E$, then the inertia group of $Q$ in 
$\varphi$ must be $G$.  This shows that the unique point $P$ with 
$\psi(P)=Q$ is ramified in the double cover $C\to \Cb$.  Since the 
characteristic of $k$ is not~$2$, the point $P$ is tamely ramified.
Likewise, $Q$ is tamely ramified in $D\to \Db$.  Thus, $\discC$ has 
order $1$ at $P$ and $\discD$ has order $1$ at~$Q$.  It follows that
\[\deg\discC - e = \deg\discC' = m\deg\discD' = m(\deg\discD-e),\]
which gives $(2g+2-e) = m(2h+2-e)$, which is what is claimed in the
left-hand column of the first table in Lemma~\ref{L:genera}.

Suppose that the characteristic of $k$ is not~$2$ and that $m$ is even.
If $Q$ is a point of~$E$, then the inertia group of $Q$ in the cover 
$\varphi$ must be $\langle\iota\beta\rangle$.  In this case we see that 
the unique point $P$ of $\Cb$ with $\psi(P) = Q$ does \emph{not} ramify
in the double cover $C\to \Cb$.  This tells us that
\[\deg\discC  = \deg\discC' = m\deg\discD' = m(e + \deg\discD),\]
which leads to the entries in the right-hand column of the first table
in Lemma~\ref{L:genera}.

Now suppose that $k$ has characteristic $2$ and that $m$ is odd, and
suppose $Q$ is a point of $E$.  Let $P$ be the unique point of $\Cb$
with $\psi(P) = Q$.  The inertia group of $Q$ in the cover 
$\varphi:C\to \Db$ must be $G$, so $P$ is ramified in the double
cover $C\to \Cb$.  As above, let $\frakp$ be the point of $C$ lying
over $P$ and let $\frakq$ be the point of $D$ lying over~$Q$.  We would
like to compare the order of $\discC$ at $P$ to the order of $\discD$
at $Q$.  Since these are locally-defined quantities, we may replace the
curves in Diagram~\eqref{EQ:prop} with their completions at $\frakp$,
$\frakq$, $P$, and $Q$, respectively.  We can then choose a uniformizer
$u$ for $\frakp$ and a uniformizer $v$ for $\frakq$ such that 
$v = u^m$.

Let $i>1$ be the valuation at $\frakp$ of $u -\iota^*u$.  Then we can 
write
\[\iota^*u = u + cu^i + \textup{(higher-order terms)},\]
and raising both sides to the $m$-th power we find that
\[\iota^*v = v + mcu^{m-1+i} + \textup{(higher-order terms)}.\]
Since the valuation of $v-\iota^*v$ at $\frakp$ is $m-1+i$, the 
valuation $j$ of $v-\iota^*v$ at $\frakq$ must be $(m-1+i)/m$.  In
other words, $i = mj - m + 1$.  If we let $I$ denote the degree of the
portion of $\discC$ supported on $\psi^{-1}(E)$, and $J$ the degree of
the portion of $\discD$ supported on $E$, then $I = mJ - (m-1)e$.

We find that
\[\deg\discC = \deg\discC' + I 
             = m\deg\discD' + mJ - (m-1)e 
             = m\deg\discD -(m-1)e, \]
so that $\deg\discC - e = m(\deg\discD  - e)$, which again leads to
$2g+2-e = m(2h+2-e)$.  This gives us the formulas in the left-hand
column of the second table in Lemma~\ref{L:genera}.

Finally we consider the case where $k$ has characteristic $2$ and $m$
is even.  As we noted before the proof of Lemma~\ref{L:ramgroups}, an
even-order automorphism of a genus-$0$ curve in characteristic $2$ must
have order~$2$, so $m=2$.  Once again, we define $E$ to be the set of
points of $\Db$ that ramify in $\psi:\Cb\to \Db$ and in the double 
cover $D\to \Db$ (so that $e=\#E$), and we define $E'$ to be the set of
points that ramify in $D\to \Db$ but not in $\psi$.  As before, we 
define $\discC'$ to be the part of $\discC$ supported on 
$\psi^{-1}(E')$ and $\discD'$ to be the part of $\discD$ supported 
on~$E'$, and as before, we have $\deg\discC' = 2\deg\discD'$.

Since $m=2$, the map $\psi$ is ramified at a single point, and 
$e\le 1$.  If $e=0$ then
\[\deg\discC = \deg\discC' = 2\deg\discD' = 2\deg\discD,\]
and it follows from the Riemann-Hurwitz formula that $2h = g-1$, as 
claimed in the second table in Lemma~\ref{L:genera}.

On the other hand, if $e=1$ we may choose isomorphisms $\Cb\cong \BP^1$
and $\Db\cong\BP^1$ so that $\psi$ corresponds to the function field 
map $x\mapsto x^2 + x$; then $\infty$ ramifies in the double cover 
$D\to \Db\cong\BP^1$.  The function field $k(D)$ of $D$ is an
Artin-Schreier extension of $k(x)$, so it contains an element $y$ not
in $k(x)$ such that $y^2 + y$ lies in $k(x)$.  The completion of $k(x)$
at $\infty$ is the ring of Laurent series in $1/x$, and in this 
completion we can write 
\[y^2 + y = a_n x^n + \sum_{i=-\infty}^{n-1} a_i x^i\]
for some integer $n$ and elements $a_i$ of $k$ with $a_n\neq 0$.
If $n = 0$ we can replace $y$ with $y+b$ for a constant $b\in k$ with
$b^2 + b = a_0$, which has the effect of replacing $n$ with nonzero 
integer.  If $n$ is even and nonzero we may replace $y$ with 
$y + \sqrt{a_n} x^{n/2}$, which has the effect of replacing $n$ by
$n/2$;  repeating this reduction, we find that we may assume that $n$
is odd, say $n=2m-1$.  If $n$ were negative the point $\infty$ would 
split in $D$, contrary to assumption, so $n$ must be positive.

Let $\frakq$ be the unique point of $D$ lying over $\infty$.  It is
easy to check that $v = y/x^m$ is a uniformizer for $\frakq$.  As
before, the order of $\deg\discD$ at $\infty$ is equal to the valuation
of $v - \iota^*v$ at $\frakq$.  Since $\iota^* v = (y+1)/x^m$, we find 
that 
\[\val_\frakq(v - \iota^*v) = \val_\frakq (1/x^m) 
                            = 2\val_\infty(1/x^m) = 2m.\]

Now consider the curve $C$, which is the fiber product of the double 
cover $D\to \Db$ with $\psi:\Cb\to \Db$.  Locally at $\infty$, we
obtain $C$ by taking the equality
\[y^2 + y = a_n x^n + \sum_{i=-\infty}^{n-1} a_i x^i\]
and replacing $x$ with $x^2 + x$.  This gives us
\[y^2 + y = a_n x^{2n} + a_nx^{2n-1} + \text{(terms in $x^i$ with $i<2n-1$)}.\]
If $n>1$, then replacing $y$ with $y + \sqrt{a_n}x^n$ gives us
\[y^2 + y = a_nx^{2n-1} + \text{(terms in $x^i$ with $i<2n-1$)},\]
and we find that $\ord_\infty\discC = 2n = 4m-2$.  But if $n=1$, the
same substitution gives
\[y^2 + y = (a_n + \sqrt{a_n})x + \text{(terms in $x^i$ with $i<1$)}.\]
In this case, $\ord_\infty \discC = 2 = 4m-2$ if $a_n \neq \sqrt{a_n}$,
and $\ord_\infty \discC = 0 = 4m-4$ otherwise.

Applying the Riemann-Hurwitz formula to the double covers $D\to \Db$
and $C\to \Cb$ and using the relation $\deg\discC' = 2\deg\discD'$, we
find that
\begin{align*}
2h &= \deg\discD' + 2m - 2\\
g  &= \begin{cases}
          \deg\discD' + 2m - 2 & \text{if $\infty$ ramifies in $C\to \Cb\cong\BP^1$;} \\
          \deg\discD' + 2m - 3 & \text{if $\infty$ is unramified in $C\to \Cb\cong\BP^1$.}
      \end{cases}
\end{align*}      
It follows that 
\[
2h = \begin{cases}
        g   & \text{if $\infty$ ramifies in $C\to \Cb\cong\BP^1$;} \\
        g+1 & \text{if $\infty$ is unramified in $C\to \Cb\cong\BP^1$.}
     \end{cases}
\]       
Clearly $g$ is even in the first case and odd in the second, so we get
the result given in the second column of the second table of
Lemma~\ref{L:genera}.
\end{proof}

\begin{remark}
One could also prove Lemma~\ref{L:genera} by using explicit equations 
and the standard formulas for the genus of a hyperelliptic curve in 
terms of its defining 
equation~\cite{goldschmidt}*{Cor.~3.6.3, Cor.~3.6.9}.  For example, 
suppose that the characteristic $p$ of the base field is not $2$, that
$m$ is odd and not divisible by $p$, and that $e=2$.  By choosing
appropriate isomorphisms $\Cb\cong\BP^1\cong\Db$ and an appropriate 
model for $D$, we can assume that $\psi$ is the map $x\mapsto x^m$ and 
that $D$ is given by $y^2 = x f(x)$ for a separable even-degree 
polynomial $f(x)$ with $f(0)\neq 0$.  Then $C$ has a singular model of
the form $y^2 = x^m f(x^m)$ and a nonsingular model of the form 
$z^2 = x f(x^m)$.  In this case one checks that $h = (\deg f )/2$ 
and $g = (m\deg f)/2$, so $2h = 2g/m$, as claimed by the lemma.
\end{remark}

\section{Proof of Theorem~\ref{T:main}.}
\label{S:proof}

In this section we prove Theorem~\ref{T:main}.  Let us begin by 
explaining the basic idea of the proof.

Since the conclusions of Theorem~\ref{T:main} are completely geometric,
we may assume that $k$ is algebraically closed.  The characteristic
polynomial $f$ of $\alpha^*$ has degree $2g$; let its complex roots be
$\zeta_1,\ldots,\zeta_{2g}$, so that the $\zeta$ are all $n$-th roots
of unity.  For each divisor $d$ of $n$ let $N_d$ denote the number of
the $\zeta$ that are primitive $d$-th roots of unity and let $M_d$
denote the number of the $\zeta$ that satisfy $\zeta^d = 1$.  Then we
have
\begin{equation}
\label{EQ:fundamental}
M_d  = \sum_{e\mid d} N_e,\qquad
N_d  = \sum_{e\mid d} \mu(d/e)M_e,\text{\quad and\quad}
f    = \prod_{d\mid n} \Phi_d^{N_d/\phi(d)},
\end{equation}
where $\mu$ is the M\"obius function, $\phi$ is the Euler 
$\phi$-function, and $\Phi_d$ is the $d$-th cyclotomic polynomial.
So to determine $f$, it is enough to determine the $M_d$.

For every divisor $d$ of $n$, let $f_d$ be the characteristic polynomial
of the automorphism $\alpha^d$ of $C$.  Then the complex roots of $f_d$
are the $d$-th powers of the complex roots of $f$, so $M_d$ is equal to
the multiplicity of $1$ as a root of $f_d$.  This multiplicity is equal
to twice the dimension of the part of the Jacobian on which $\alpha^d$
acts trivially, and this dimension is equal to  the genus of the
quotient of $C$ by $\langle\alpha^d\rangle$.  We see that computing
$M_d$ is equivalent to computing the genus of this quotient curve.  If
the hyperelliptic involution $\iota$ lies in $\langle\alpha^d\rangle$
then the genus of the quotient is~$0$; if not, then the genus is
determined by Lemma~\ref{L:genera}.  To prove the theorem, all we must
do is verify that the values of $M_d$ predicted by the putative 
characteristic polynomials given in the theorem agree with the values 
we compute by applying Lemma~\ref{L:genera}.  

We consider the four statements of the theorem in turn.

\emph{Proof of Statement}~\eqref{itemb}:
In this case our assumption is that $n=\nb$ and $\nb$ is odd.  For each
divisor $d$ of $n$ let $D_d$ be the quotient of $C$ by 
$\langle\alpha^d\rangle$.  Then we have a diagram
\begin{equation}
\label{EQ:itemb}
\xymatrix{
C \ar[d]^2 \ar[rr]^{n/d} &&  D_d\ar[d]^2\ar[rr]^d && D_1\ar[d]^2\\
\Cb\ar[rr]^{n/d}         && \Db_d       \ar[rr]^d && \Db_1.
}
\end{equation}
As in Section~\ref{S:quotients}, we see that the map $\varphi$ from $C$
to $\Db_1$ is a Galois cover with group 
$G = (\BZ/n\BZ)\times(\BZ/2\BZ)$.  Let $\varphi_d$ be the map from $C$
to $\Db_d$, let $E_d$ be the set of points of $\Db_d$ that ramify going 
up to $D_d$ and going up $\Cb$, and let $e_d = \#E_d$.  We will show 
that $e_d$ is determined by~$e_1$.

Since $\nb$ is odd, Lemma~\ref{L:ramgroups} tells us that the inertia
group of a point of $\Db_1$ in the extension $\varphi:C\to\Db_1$ is
either trivial, or $\langle\iota\rangle$, or $\langle\alpha\rangle$, 
or all of $G$.  A point in $E_1$ must have ramification group $G$, 
and so must lie under a unique point in $E_d$.  Likewise, any point
in $E_d$ must lie over a point of $\Db_1$ that has ramification 
group~$G$, and that therefore lies in $E_1$. Thus, for every $d$ we 
have $e_d = e_1$.

If $e_1=0$ then Lemma~\ref{L:genera} shows that $M_d = (2g+2)d/\nb - 2$
for all $d$.  In particular we see that $\nb$ divides $2g+2$.  Also, we
check that the polynomial $f=(x^\nbar-1)^{(2g+2)/\nbar}/(x-1)^2$ 
produces the correct values of $M_d$.  If $e_1=1$ then we have
$M_d = (2g+1)d/\nb - 1$, so $\nb$ divides $2g+1$, and the polynomial
$f=(x^\nbar-1)^{(2g+1)/\nbar}/(x-1)$ gives the correct values of~$M_d$.
Finally, if $e_1 = 2$ then $M_d = 2gd/\nb$, so that $\nb$ divides $2g$,
and the polynomial $f = (x^\nbar - 1)^{2g/\nbar}$ produces the required
values of $M_d$.

\emph{Proof of Statement}~\eqref{itema}:
In this case $n=2\nb$ and $\nb$ is odd, and we see that 
$\iota = \alpha^\nbar$.  Let $\alpha_0 = \iota\alpha$, so that 
$\alpha_0$ has order $\nb$ and induces an automorphism of order $\nb$
on~$\Cb$.  Then Statement~\eqref{itemb} tells us the characteristic 
polynomial $f_0$ of $\alpha_0^*$; furthermore, since 
$\alpha^* = -\alpha_0^*$, we have $f(x) = f_0(-x)$.  This agrees with
what is claimed in Statement~\eqref{itema}.

\emph{Proof of Statement}~\eqref{itemd}:
In this case $n=\nb$ and $\nb$ is even, and the analysis is very much
like that for Statement~\eqref{itemb}.  For every divisor $d$ of $n$
we let $D_d$ be the quotient of $C$ by $\langle\alpha^d\rangle$, and
Diagram~\eqref{EQ:itemb} is again a diagram of Galois extensions, with
the total Galois group $G$ being $(\BZ/n\BZ)\times(\BZ/2\BZ)$.  However,
since $n$ is even, $G$ is no longer a cyclic group.  As before, we let
$\varphi$ be the map from $C$ to $\Db_1$, we let $\varphi_d$ be the map
from $C$ to $\Db_d$, we let $E_d$ be the set of points of $\Db_d$ that
ramify going up to $D_d$ and going up to $\Cb$, and we let 
$e_d = \#E_d$.

Let us first consider the case in which the characteristic of the base 
field is not equal to $2$.  Then according to Lemma~\ref{L:ramgroups}, 
the ramification group of a point $Q$ of $\Db_1$ in the cover $\varphi$
is either trivial, the group $\langle\iota\rangle$, the group 
$\langle\alpha\rangle$, or the group $\langle\iota\alpha\rangle$.  If a
point has one of the first two inertia groups it will not lie in $E_d$,
because it is not ramified in the bottom row of
Diagram~\eqref{EQ:itemb}.  If a point has inertia group 
$\langle\alpha\rangle$, then it will not lie in $E_d$ because it is not
ramified in the extension $D_d\to\Db_d$.  But if a point has inertia
group $\langle\iota\alpha\rangle$, it will lie in $E_d$ if $d$ is odd, 
and will not lie in $E_d$ if $d$ is even.

So when the characteristic of the base field is not~$2$, we see once
again that the value of $e_d$ is determined by the value of $e_1$: We 
have $e_d=e_1$ if $d$ is odd, and $e_d=0$ if $d$ is even.  If $d$ is
odd then $n/d$ is even, so Lemma~\ref{L:genera} tells us that 
$M_d = (2g+2)d/n -2 + e_1$.  (Note that since $M_1$ is twice the genus
of~$D_1$, we find that $n$ divides $2g+2$ and that the parity of $e_1$
is equal to the parity of $(2g+2)/n$.)  If $d$ is even then 
Lemma~\ref{L:genera} shows that $M_d = (2g+2)d/n - 2$.  These values
for $M_d$ are consistent with 
\[
f = \begin{cases}
    \displaystyle \frac{(x^n-1)^{(2g+2)/n}}{(x-1)^2} & \text{if $e_1 = 0$;}\\[1em]
    \displaystyle \frac{(x^n-1)^{(2g+2)/n}}{(x^2-1)} & \text{if $e_1 = 1$;}\\[1em]
    \displaystyle \frac{(x^n-1)^{(2g+2)/n}}{(x+1)^2} & \text{if $e_1 = 2$,}
    \end{cases}
\]    
so these must be the correct values of $f$.

As we noted above, the parity of $(2g+2)/n$ is equal to that of $e_1$.
Thus, if $(2g+2)/n$ is odd then $e_1=1$, and we find the value of $f$
given in Statement~\eqref{itemd1}.  If $(2g+2)/n$ is even then $e_1$ is
either $0$ or $2$, and we find that $f$ must have one of the two values
given in Statement~\eqref{itemd2}.

Finally, we turn to the case in which the base field has
characteristic~$2$.  In this case we must have $n=2$, so we only have 
to determine the value of $M_1$ (since we already know that 
$M_2 = 2g$).  But Lemma~\ref{L:genera} tells us the possibilities for 
this value:  If $g$ is even then $M_1 = g$, while if $g$ is odd then
$M_1$ is either $g-1$ or $g+1$.  We check that the values of $f$ given
in Statements~\eqref{itemd1} and~\eqref{itemd2} agree with these values
of $M_1$ and~$M_2$.

\emph{Proof of Statement}~\eqref{itemc}:
In this case $n=2\nb$ and $\nb$ is even, and we have 
$\iota = \alpha^\nbar$.  Taking the quotient of $C$ by 
$\langle\alpha\rangle$ gives us a Galois extension
\begin{equation}
\label{EQ:itemc}
\xymatrix{
C \ar[rr]^{2} &&  \Cb\ar[rr]^\nbar && \Db
}
\end{equation}
with group $G=\BZ/n\BZ$, where $\Cb$ and $\Db$ are curves of genus~$0$.

Consider a point $Q$ of $\Db$ that ramifies going up to~$\Cb$.  Then
$Q$ must be totally ramified in this extension, so the inertia group of
$Q$ in the total extension $C\to\Db$ is a subgroup of $G$ that surjects
onto the Galois group of $\Cb\to\Db$.  The only such subgroup is $G$
itself, so any point of $\Db$ that ramifies going up to $\Cb$ must 
ramify totally in the extension $C\to\Db$.

We see that if $d$ is a divisor of $n$ such that $n/d$ is even, then 
$\iota$ lies in the subgroup $\langle\alpha^d\rangle$, the genus of the
quotient of $C$ by this subgroup is $0$, and $M_d=0$.  If $d$ is a 
divisor of $n$ such that $n/d$ is odd, let $e_d$ be the number $e$ 
associated to $\alpha^d$ as in Lemma~\ref{L:genera}; then $e_d$ is 
equal to the number of points of $\Db$ that ramify in the degree-$\nb$
extension $\Cb\to\Db$, and this value is either $1$ or $2$, depending
on whether or not $\nb$ is divisible by the characteristic of the base 
field.

Suppose the characteristic of the base field is not equal to~$2$.  
Then, since $\nb$ is even, the degree-$\nb$ map $\Cb\to\Db$ of 
genus-$0$ curves does not give an Artin-Schreier extension of function
fields;  rather, it gives a Kummer extension, and it follows that there
are two points of $\Db$ that totally ramify going up to $C$.  Any other
points of $\Db$ that ramify going up to $C$ must have ramification 
groups of order $2$.  If there are $r$ of these points, then the 
Riemann-Hurwitz formula for Galois extensions tells us that 
\[2g-2 = n \left(-2 + 2(1 - 1/n) + r(1-1/2)\right)= -2 + r\nb,\]
and it follows that $\nb$ divides $2g$.

Also, since we have two points of $\Db$ that ramify totally in 
$C\to\Db$, we see that $e_d=2$ whenever $n/d$ is odd, and it follows
from Lemma~\ref{L:genera} that $M_d = 2g d/n$ when $n/d$ is odd.
Combining this with the observation that $M_d=0$ when $n/d$ is even,
we see that the polynomial $f = (x^\nbar+1)^{2g/\nbar}$ gives the
correct values of $M_d$.

Now suppose that the characteristic of the base field is equal to~$2$.
Then $\nb$ must be equal to $2$, and $\alpha$ has order $4$.  The 
diagram~\eqref{EQ:itemc} shows that the quotient of $C$ by
$\langle\alpha\rangle$ has genus $0$, the quotient of $C$ by 
$\langle\alpha^2\rangle$ has genus $0$, and the quotient of $C$ by 
$\langle\alpha^4\rangle$ has genus $g$.  Thus $M_1=M_2=0$ and $M_4=2g$.
The polynomial that gives rise to these values of $M_d$ is $(x^2+1)^g$,
which is the polynomial given in Statement~\eqref{itemc}.
\qed




\end{document}